\documentclass[10pt]{article}
\font\smallit=cmti10
\font\smalltt=cmtt10

\usepackage{amssymb,latexsym,amsmath,epsfig,amsthm} 
\usepackage{empheq}
\usepackage{ascmac}
\makeatletter

\renewcommand\section{\@startsection {section}{1}{\z@}
{-30pt \@plus -1ex \@minus -.2ex}
{2.3ex \@plus.2ex}
{\normalfont\normalsize\bfseries\boldmath}}

\renewcommand\subsection{\@startsection{subsection}{2}{\z@}
{-3.25ex\@plus -1ex \@minus -.2ex}
{1.5ex \@plus .2ex}
{\normalfont\normalsize\bfseries\boldmath}}

\renewcommand{\@seccntformat}[1]{\csname the#1\endcsname. }

\makeatother
\newtheorem{theorem}{Theorem}
\newtheorem{lemma}{Lemma}
\newtheorem{conjecture}{Conjecture}

\theoremstyle{definition}
\newtheorem{defn}{Definition}[section]
\newtheorem{rem}{Remark}[section]
\newtheorem{exam}{Example}[section]
\newtheorem{pict}{Figure}[section]

\begin{document}

\begin{center}
\uppercase{\bf Previous Player's Positions of Impartial Three-Dimensional Chocolate-Bar Games }
\vskip 20pt
{\bf Ryohei Miyadera }\\
{\smallit Keimei Gakuin Junior and High School, Kobe City, Japan}. \\
{\tt runnerskg@gmail.com}
\vskip 10pt
{\bf Hikaru Manabe}. \\
{\smallit Keimei Gakuin Junior and High School, Kobe City, Japan}. \\
{\tt urakihebanam@gmail.com}
\vskip 10pt
{\bf Shunsuke Nakamura}. \\
{\smallit Independent Researcher, Tokyo, Japan }. \\
{\tt nakamura.stat@gmail.com}

\end{center}
\vskip 20pt
\centerline{\smallit Received: , Revised: , Accepted: , Published: } 
\vskip 30pt

\pagestyle{myheadings} 
\markright{\smalltt INTEGERS: 19 (2019)\hfill} 
\thispagestyle{empty} 
\baselineskip=12.875pt 
\vskip 30pt

\centerline{\bf Abstract}
\noindent
In this study, we investigate three-dimensional chocolate bar games, which are variants of the game of Chomp. A three-dimensional chocolate bar is a three-dimensional array of cubes in which a bitter cubic box is present in some part of the bar. Two players take turns and cut the bar horizontally or vertically along the grooves. The player who manages to leave the opponent with a single bitter block is the winner.   
We consider the $\mathcal{P}$-positions of this game, where the $\mathcal{P}$-positions are positions of the game from which the previous player (the player who will play after the next player) can force a win, as long as they play correctly at every stage.
We present sufficient conditions for the case when the position $\{p,q,r\}$ is a $\mathcal{P}$-position if and only if $(p-1) \oplus (q-1) \oplus (r-1)$, where
$p, q$, and $r$ are the length, height, and width of the chocolate bar, respectively.
 \pagestyle{myheadings} 
 \markright{\smalltt \hfill} 
 \thispagestyle{empty} 
 \baselineskip=12.875pt 
 \vskip 30pt

\section{Introduction}\label{introductionsection}

Chocolate bar games are variants of the Chomp game presented in \cite{gale}.
A two-dimensional chocolate bar is a two-dimensional array of squares in which a bitter square printed in black is present in some part of the bar. See the chocolate bars in Figure \ref{two2dchoco}. 

A three-dimensional chocolate bar is a three-dimensional array of cubes in which a bitter cubic box printed in black is present in some parts of the bar. Figure \ref{two3dchoco} displays examples of three-dimensional chocolate bars. Games involving these chocolate bars may be defined as follows. 

\begin{defn}\label{definitionofchoco}
(i) Two-dimensional chocolate bar game: Each player in turn breaks the bar in a straight line along the grooves and eats the broken piece. The player who manages to leave the opponent with a single bitter block (black block) is the winner. \\
(ii) Three-dimensional chocolate game: The rules are the same as in (i), except that the chocolate is cut horizontally or vertically along the grooves. Examples of cutting three-dimensional chocolate bars are shown in Figure \ref{3dcut}.
\end{defn}

\begin{pict}
\begin{tabular}{cc}
\begin{minipage}{.33\textwidth}
\centering
\includegraphics[height=1.4cm]{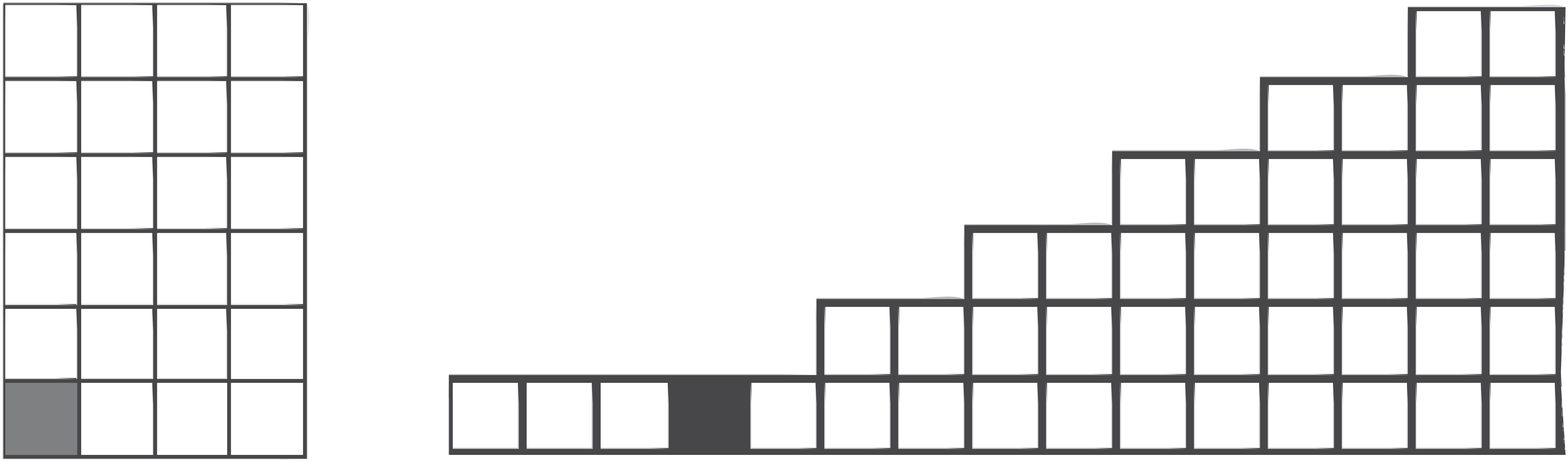}
\label{two2dchoco}
\end{minipage}
\begin{minipage}{.33\textwidth}
\centering
\includegraphics[height=1.8cm]{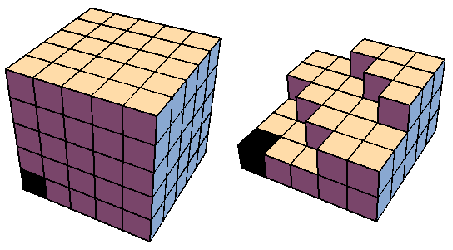}
\label{two3dchoco}
\end{minipage}
\end{tabular}
\end{pict}

\begin{exam}
Three methods of cutting a three-dimensional chocolate bar.\\
\begin{pict}
\includegraphics[height=2.7cm]{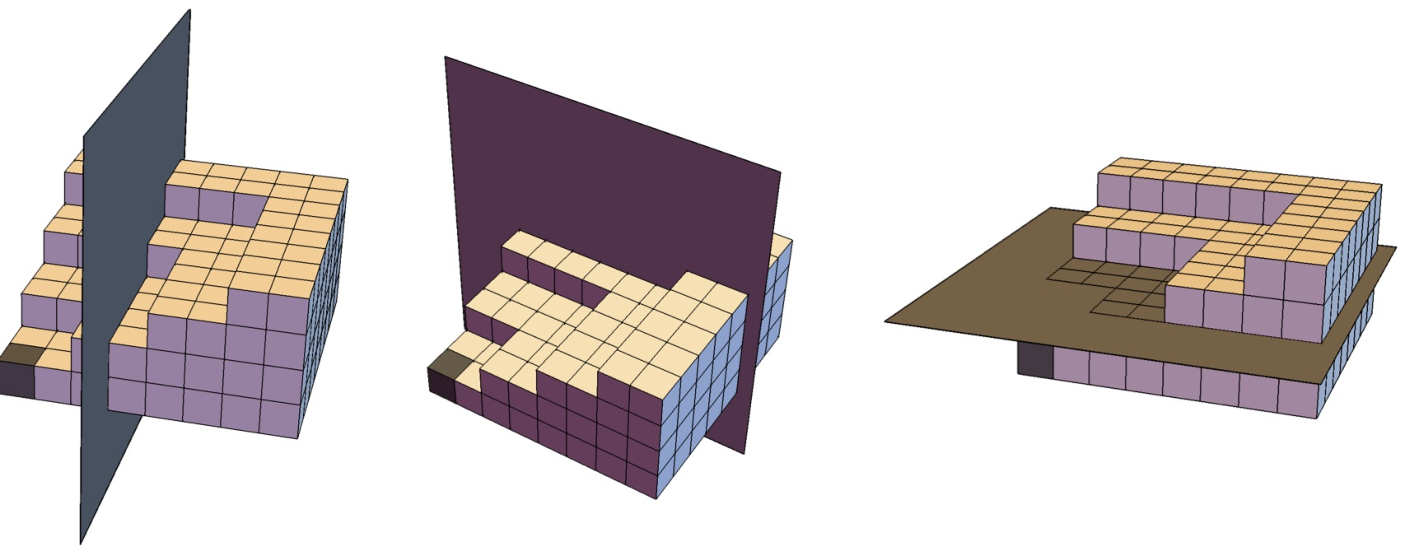}
\label{3dcut}
\end{pict}
\noindent
\end{exam}

For completeness, we briefly review some of the necessary concepts of combinatorial game theory; refer to \cite{lesson} for greater detail. Let $Z_{\ge 0}$ and $N$ be sets of non-negative integers and  natural numbers, respectively.

\begin{defn}\label{definitionfonimsum11}
	Let $x$ and $y$ be non-negative integers. Expressing them in Base 2 yields 
$x = \sum_{i=0}^n x_i 2^i$ and $y = \sum_{i=0}^n y_i 2^i$ with $x_i,y_i \in \{0,1\}$.
	We define nim-sum $x \oplus y$ as:
	\begin{equation}
		x \oplus y = \sum\limits_{i = 0}^n {{w_i}} {2^i}.
	\end{equation}
	where $w_{i}=x_{i}+y_{i} \ (\bmod\ 2)$.
\end{defn}

As chocolate bar games are impartial games without draws, only two outcome classes are possible.
\begin{defn}\label{NPpositions}
	$(a)$ A position is called a $\mathcal{P}$-\textit{position}, if it is a winning position for the previous player (the player who just moved), as long as he/she plays correctly at every stage.\\
	$(b)$ A position is called an $\mathcal{N}$-\textit{position}, if it is a winning position for the next player, as long as he/she plays correctly at every stage.
\end{defn}

\begin{defn}\label{defofmexgrundy2}
$(i)$ For any position $\mathbf{p}$ of game $\mathbf{G}$, there is a set of positions that can be reached by precisely one move in $\mathbf{G}$, which we denote as \textit{move}$(\mathbf{p})$. \\
$(ii)$ The \textit{minimum excluded value} ($\textit{mex}$) of a set $S$ of non-negative integers is the least non-negative integer that is not in S. \\
$(iii)$ Each position $\mathbf{p}$ of an impartial game has an associated Grundy number, and we denote this as $\mathcal{G}(\mathbf{p})$.\\
The Grundy number is recursively defined by $G(\mathbf{p}) = \textit{mex}\{G(\mathbf{h}): \mathbf{h} \in move(\mathbf{p})\}.$
\end{defn}

\begin{theorem}\label{theoremofsumg2}
For any position $\mathbf{p}$ of the game, 
	$\mathcal{G}(\mathbf{p}) =0$ if and only if $\mathbf{p}$ is a $\mathcal{P}$-position.
\end{theorem}

The original two-dimensional chocolate bar introduced by Robin \cite{robin} is the chocolate shown on the left-hand side in Figure \ref{two2dchoco}.
Because the horizontal and vertical grooves are independent, an $m \times n$ rectangular chocolate bar is equivalent to the game of Nim, which includes heaps of $m-1$ and $n-1$ stones, respectively. Therefore, the chocolate $6 \times 4$ bar game shown on the left-hand side of Figure \ref{two2dchoco} is mathematically the same as Nim, which includes heaps of $5$ and $3$ stones, respectively.
It is well known that the Grundy number of the Nim game with heaps of $m-1$ stones and $n-1$ stones is $(m-1) \oplus (n-1)$; therefore, the Grundy number of the $m \times n$ rectangular bar is $(m-1) \oplus (n-1)$.
Robin \cite{robin} also presented a cubic chocolate bar, as shown on the left-hand side of Figure \ref{two3dchoco}.
It can be easily determined that this $5 \times 5 \times 5$ three-dimensional chocolate bar is mathematically the same as Nim with heaps of $4$, $4$, and $4$ stones, and the Grundy number of this cuboid bar is $4 \oplus 4 \oplus 4$.
It is then natural to ask the following question.

\vspace{0.5cm}
\noindent
\bf{Question 1. \ }\normalfont
\textit{What is the necessary and sufficient condition whereby a three-dimensional chocolate bar may have a Grundy number $(p-1) \oplus (q-1) \oplus (r-1)$, where $p, q$, and $r$ are the length, height, and width of the bar, respectively?}
\normalfont
\vspace{0.5cm}

Although the authors answered this question for two-dimensional chocolate bars in \cite{jgame} and for the three-dimensional case in \cite{ns3d}, the results of these studies are omitted here.

When the Grundy number of a chocolate bar with $p, q$, and $r$ as the length, height, and width, respectively, is
$(p-1) \oplus (q-1) \oplus (r-1)$, the position is a $\mathcal{P}$-position if and only if 
$(p-1) \oplus (q-1) \oplus (r-1)=0$ for the chocolate bar.

Therefore, it is natural to ask the following question.

\vspace{0.5cm}
\noindent
\bf{Question 2. \ }\normalfont
\textit{Under what condition may a three dimensional chocolate bar with $p, q$, and $r$ as the length, height, and width, respectively, have a $\mathcal{P}$-position if and only if $(p-1) \oplus (q-1) \oplus (r-1)=0$?}
\normalfont
\vspace{0.5cm}

In the remainder of this paper, we present a sufficient condition for which Question 2 may be answered. 

Determining the necessary and sufficient conditions for this question is a very difficult unsolved problem considered by the authors.
We suppose that the difficulty of presenting the necessary and sufficient conditions arises from the fact that there are many kinds of sufficient conditions. For more information, see Theorems \ref{theoremforoddk} in Section \ref{sub4mone} and Conjecture \ref{theoremmanabe} in Section \ref{others}.

 We now define a three-dimensional chocolate bar. 

\begin{defn}\label{definitionoffunctionf3d}
	Suppose that $f(u,v)\in Z_{\geq0}$ for $u,v \in Z_{\geq0}$. $f$ is said to monotonically increase if $f(u,v) \leq f(x,z)$ for $x,z,u,v \in Z_{\geq0}$ with $u \leq x$ and $v \leq z$.
\end{defn} 

\begin{defn}\label{defofbarwithfunc3d}
	Let $f$ be the monotonically increasing function in Definition \ref{definitionoffunctionf3d}.\\ 
	Let $x,y,z \in Z_{\geq0}$.
	The three-dimensional chocolate bar comprises a set of $1 \times 1 \times 1$ boxes. 
For $u,w \in Z_{\geq0}$ such that $u \leq x$ and $w \leq z$, the height of the column at position $(u,w)$ is $ \min (f(u,w),y) +1$. 
There is a bitter box at position $(0,0)$.
We denote this chocolate bar as $CB(f,x,y,z)$. Note that $x+1, y+1$, and $z+1$ are the length, height, and width of the bar, respectively.
\end{defn}

\begin{exam}
Here, we let $f(x,z)$ $= \lfloor \frac{x+z}{3}\rfloor$, where $ \lfloor \ \rfloor$ is the floor function, and we present several examples of  $CB(f,x,y,z)$.
\begin{pict}
	\centering
\includegraphics[height=3cm]{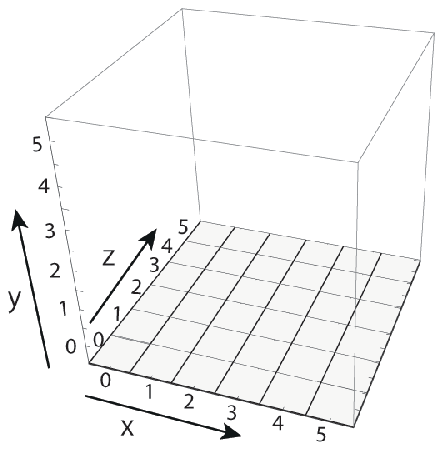}
\label{coordinate3d}
\end{pict}

\begin{pict}
\begin{tabular}{cc}
\begin{minipage}{.33\textwidth}
\centering
\includegraphics[height=2.5cm]{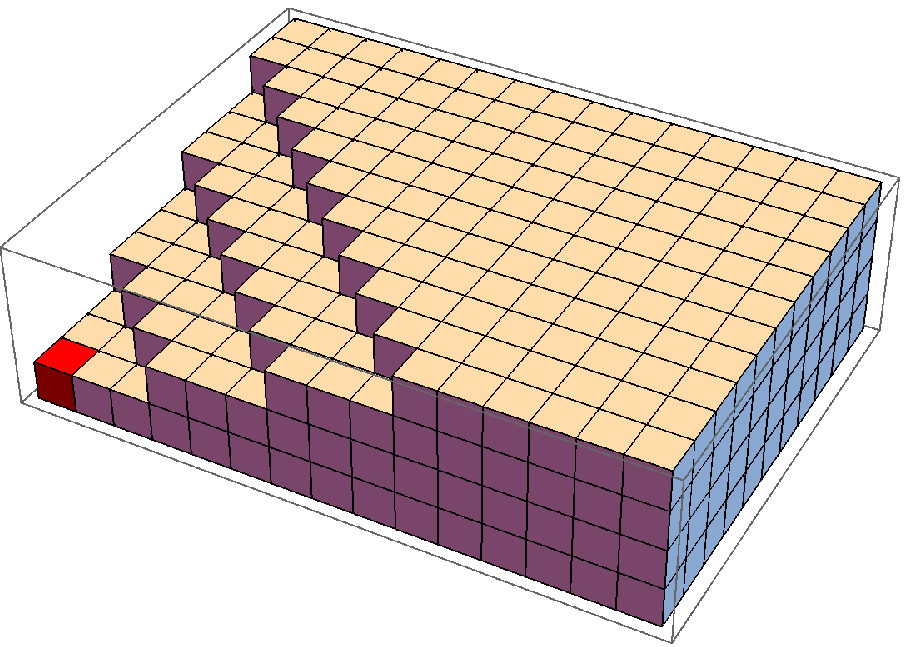}

$CB(f,14,3,10)$
\label{f14310}
\end{minipage}
\begin{minipage}{.33\textwidth}
\centering
\includegraphics[height=2.5cm]{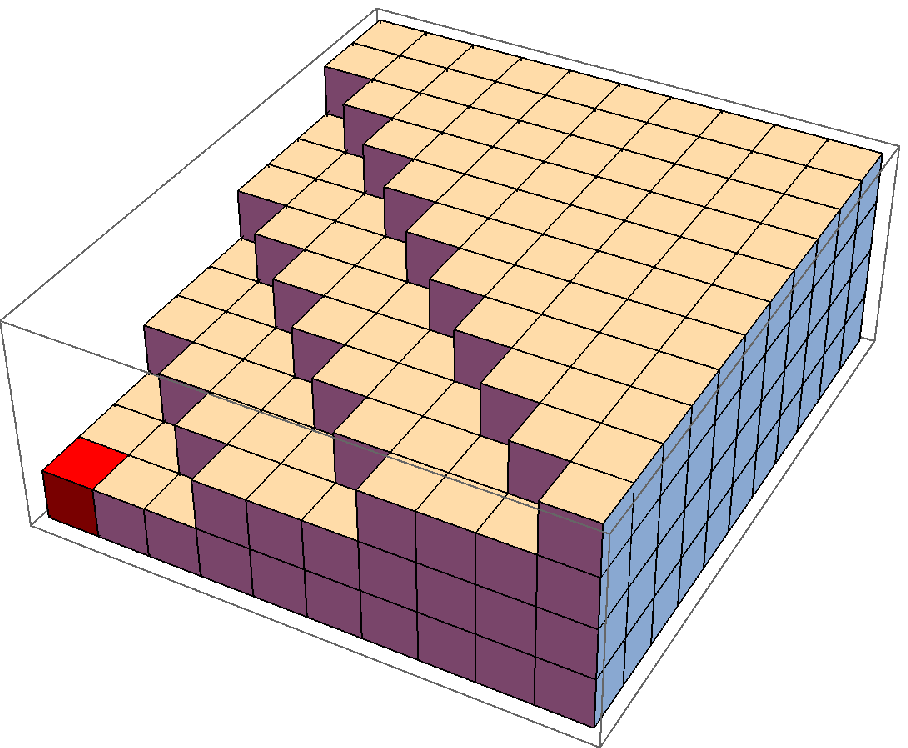}

$CB(f,9,3,10)$
\label{f9310}
\end{minipage}	
\end{tabular}
\end{pict}

\begin{pict}
\begin{tabular}{cc}
\begin{minipage}{.33\textwidth}
	\centering
\includegraphics[height=3cm]{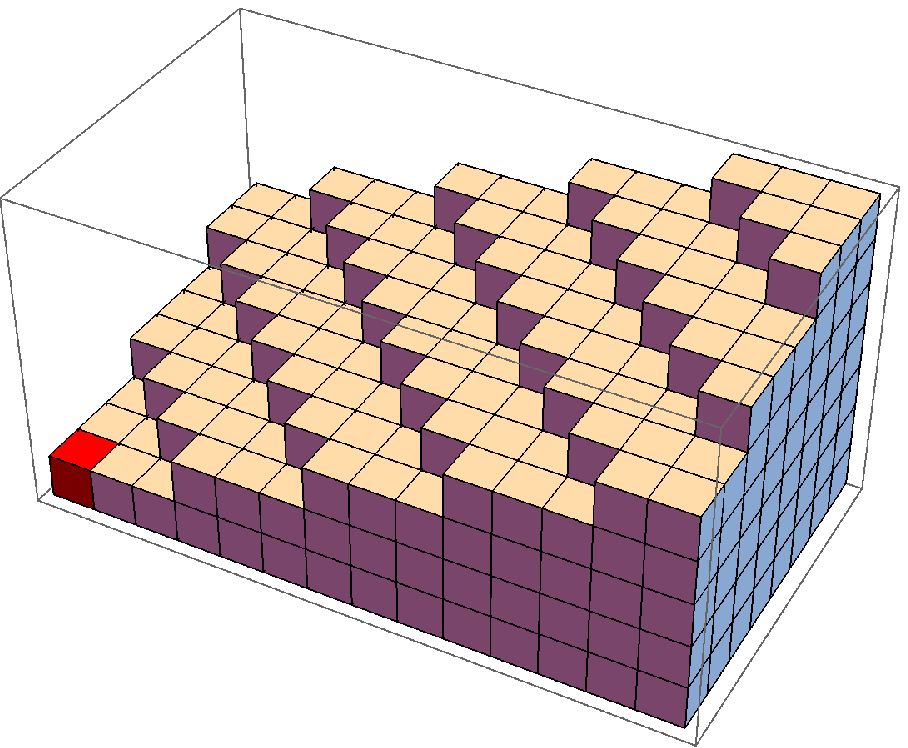}

$CB(f,13,6,7)$
\label{f1367}
\end{minipage}
\begin{minipage}{.33\textwidth}
		\centering
\includegraphics[height=2.2cm]{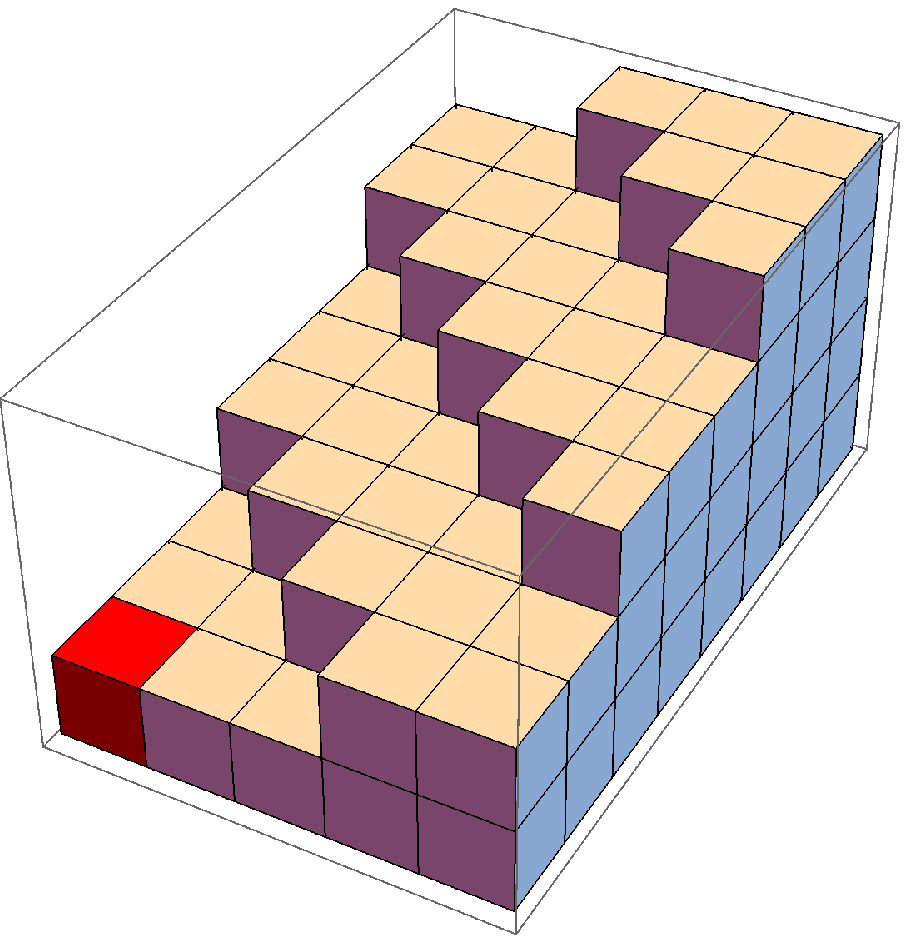}

$CB(f,4,3,7)$ 
\label{f437}
\end{minipage}	
\end{tabular}
\end{pict}
\end{exam}

Next, we define $move_f(\{x, y, z\})$ in Definition \ref{movefor3dimension}. $move_f(\{x, y, z\})$ is a set that contains all of the positions that can be reached from position $\{x, y, z\}$ in one step (directly).

\begin{defn}\label{movefor3dimension}
	For $x,y,z \in Z_{\ge 0}$, we define 
\begin{align}
& move_f(\{x,y,z\})=\{\{u,\min(f(u,z),y),z \}:u<x \} \cup \{\{x,v,z \}:v<y \}   \nonumber \\
& \cup \{ \{x,\min(y, f(x,w) ),w \}:w<z \}, \text{ where \ } u,v,w \in Z_{\ge 0}.\nonumber
\end{align}	
\end{defn}

\begin{rem}
Definition \ref{movefor3dimension} shows how to reduce the coordinates of the chocolate bar by cutting, and in Example \ref{chococute}, we provide concrete examples of reducing the coordinates.
\end{rem}

\section{When $f(x,z) = \lfloor \frac{x+z}{k}\rfloor$ for $k =4m+3$}\label{sub4mone}
Let $k = 4m + 3$ for some $m \in Z_{\geq 0}$.
Let $x = \sum_{i=0}^n x_i 2^i$, $y = \sum_{i=0}^n y_i 2^i$ and $z = \sum_{i=0}^n z_i 2^i$ for some $n \in Z_{\ge 0}$ and $x_i,y_i,z_i \in \{0,1\}$.
Throughout this section, we assume that
\begin{equation}
f(x,z) = \lfloor \frac{x+z}{k}\rfloor.
\end{equation}	
Before we prove several lemmas, we first consider the procedures provided in Example \ref{chococute}. Although this example is lengthy, the proofs of the lemmas are difficult to understand without first considering this example.

\begin{exam}\label{chococute}
Let $f(x,z)$ $= \lfloor \frac{x+z}{3}\rfloor$. \\
$(i)$ We begin with the chocolate bar shown in Figure \ref{f14310}. If the first coordinate $x = 14$ is reduced to $u=9$ by cutting the 
chocolate bar $CB(f,14,3,10)$ shown in Figure \ref{f14310}, by Definition \ref{movefor3dimension} the second coordinate will be $\min(f(u,z),y)$
$= \min(f(9,10),3)$ $= \min(\lfloor \frac{19}{3}\rfloor,3)$ $= \min(6,3)$ = 3. Therefore,
we can reduce $x = 14$ to $u=9$ without affecting the second coordinate $3$, which is the height of the chocolate bar, and 
we obtain the chocolate bar $CB(f,9,3,10)$ shown in Figure \ref{f9310} (i.e., $\{9,3,10\} \in move_f(\{14,3,10\})$). 

\begin{table}
\begin{tabular}{cc}
\begin{minipage}{.5\textwidth}
	\centering
\begin{tabular}{|c|c|c|c|} \hline
\text{  \ } & $x=14$ & $y=3$ & $z=10$ \\ \hline
$2^3=8$  & $x_3=1$ & $y_3=0$ & $z_3=1$ \\ \hline
$2^2=4$ & $x_2=1$ & $y_2=0$ & $z_2=0$ \\ \hline
$2^1=2$ & $x_1=1$ & $y_1=1$ & $z_1=1$ \\ \hline
$2^0=1$ & $x_0=0$ & $y_0=1$ & $z_0=0$ \\ \hline
\end{tabular}
 \caption{$CB(f,14,3,10)$ }\label{3D1981}
\end{minipage}
\begin{minipage}{.5\textwidth}
\centering
\begin{tabular}{|c|c|c|c|} \hline
 \text{  \ } & $u=9$ & $y=3$ & $z=10$ \\ \hline
$2^3=8$  & $u_3=1$ & $y_3=0$ & $z_3=1$ \\ \hline
$2^2=4$ & $u_2=0$ & $y_2=0$ & $z_2=0$ \\ \hline
$2^1=2$ & $u_1=0$ & $y_1=1$ & $z_1=1$ \\ \hline
$2^0=1$ & $u_0=1$ & $y_0= 1$ &$z_0=0$ \\ \hline
\end{tabular}
  \caption{$CB(9,3,10)$ }\label{3D19}
\end{minipage}	
\end{tabular}
\end{table}
\noindent
$(ii)$ We begin with the chocolate bar in Figure \ref{f1367}. If the first coordinate $x = 13$ is reduced to $u=4$ by cutting the chocolate bar $CB(f,13,6,7)$ in Figure \ref{f1367} by Definition \ref{movefor3dimension}, the second coordinate will be $\min(f(u,z),y)$
$= \min(f(4,7),6)$ $= \min(\lfloor \frac{11}{3}\rfloor,6) =\min(3,6) =3$. Therefore, the second coordinate, $6$, which is the height of the chocolate bar, will be reduced to $3$. Then, we obtain the chocolate bar shown in Figure \ref{f437} (i.e., $\{4,3,7\} \in move_f(\{13,6,7\})$). 
\begin{table}
\begin{tabular}{cc}
\begin{minipage}{.5\textwidth}
\centering
\begin{tabular}{|c|c|c|c|} \hline
 \text{  \ } & $x=13$ & $y=6$ & $z=7$ \\ \hline
$2^3=8$ & $x_3=1$ & $y_3=0$ & $z_3=0$ \\ \hline
$2^2=4$ & $x_2=1$ & $y_2=1$ & $z_2=1$ \\ \hline
$2^1=2$ & $x_1=0$ & $y_1=1$ & $z_1=1$ \\ \hline
$2^0=1$ & $x_0=1$ & $y_0=0$ & $z_0=1$ \\ \hline
\end{tabular}
 \caption{$CB(f,13,6,7)$ }\label{3D1982}
\end{minipage}
\begin{minipage}{.5\textwidth}
		\centering
\begin{tabular}{|c|c|c|c|} \hline
 \text{  \ } & $x=4$ & $y=3$ & $z=7$ \\ \hline
$2^3=8$ & $u_3=0$ & $v_3=0$ & $z_3=0$ \\ \hline
$2^2=4$ & $u_2=1$ & $v_2=0$ & $z_2=1$ \\ \hline
$2^1=2$ & $u_1=0$ & $v_1=1$ & $z_1=1$ \\ \hline
$2^0=1$ & $u_0=0$ & $v_0=1$ & $z_0=1$ \\ \hline
\end{tabular}
  \caption{$CB(f,4,3,7)$ }\label{3D19b}
\end{minipage}	
\end{tabular}
\end{table}
\noindent
$(iii)$ The procedures presented in $(i)$ and $(ii)$ are good examples of moving to a position whose nim-sum is 0 from a position whose nim-sum is not 0.

In $(i)$, $14 \oplus 3 \oplus 10 = 7$, and suppose that the player wants to move to a position whose nim-sum is 0. First, let $u_3= x_3 = 1$. Next, 
reduce $x_2 =1$ to $u_2 = 0$. Note that 
\begin{equation}\label{xbigeru0}
x = \sum_{i=0}^3 x_i 2^i =  2^3 +  2^2 +  2 >  2^3 + 0 \times 2^2 + u_1 \times 2 + u_0 = \sum_{i=0}^3 u_i 2^i =u
\end{equation}
regardless of the values of $u_1, u_0$. Then, reduce $x_1$ to $u_1=0$ and increase $x_0 = 0$ to $u_0 =1$. Note that by considering (\ref{xbigeru0}), one can choose any value for $u_1, u_0$. Then, we obtain the position $\{9,3,10\}$ such that $9 \oplus 3 \oplus 10 = 0$.

In $(ii)$, $13 \oplus 6 \oplus 7 = 12$, and suppose that it is desired to move to a position whose nim-sum is 0. First, we have to reduce $x_3=1$ to $u_3=0$. 
Because $\{x_2,y_2,z_2\} = \{1,1,1\}$ and $1 \oplus 1 \oplus 1 \ne 0 \ (\mod 2)$, we may let 
$\{u_2,y_2,z_2\} = \{0,1,1\}$ or $\{u_2,v_2,z_2\} = \{1,0,1\}$ by reducing $y$ to $v$. Note that once we reduce $x$, we cannot reduce $z$.

If 
\begin{equation}\label{firstcase}
\{u_2,y_2,z_2\} = \{0,1,1\},
\end{equation}
we have
\begin{align}
& f( u, z )\\ \nonumber
& = f( \sum_{i=0}^3 u_i 2^i, \sum_{i=0}^3 z_i 2^i )\\ \nonumber
& =  f(0 \times 2^3 + 0 \times 2^2 + u_1 2^1 + u_0 2^0, 7) \\ \nonumber
& =   \lfloor \frac{7+u_1 2^1 + u_0 2^0}{3}\rfloor  \leq \lfloor \frac{10}{3}\rfloor=3, \label{xbigeru}
\end{align}
regardless of the values of $u_1, u_0$. We then have $f(u,z) < 4 = y_2 2^2 \leq y$.
Therefore, by Definition \ref{movefor3dimension} (\ref{firstcase}) leads to a contradiction.

We should then let 

\begin{equation}\label{secondcase}
\{u_2,v_2,z_2\} = \{1,0,1\},
\end{equation}
by simultaneously reducing $x$ and $y$.

Next, we let $\{u_1,v_1,z_1\} = \{0,1,1\}$ or $\{u_1,v_1,z_1\} = \{1,0,1\}$.

If $\{u_1,v_1,z_1\} = \{1,0,1\}$, by (\ref{secondcase})
\begin{align}
& f(u, z)\\ \nonumber
& = f( \sum_{i=0}^3 u_i 2^i, \sum_{i=0}^3 z_i 2^i )\\ \nonumber
& =  f( 0 \times 2^3 + 2^2 +  2^1 + u_0 2^0, 7) \\ \nonumber
& =   \lfloor \frac{13 + u_0 2^0}{3}\rfloor  \geq \lfloor \frac{13}{3}\rfloor \\ \nonumber
& =4 > 1 \geq \sum_{i=0}^3 v_i 2^i = 0 \times 2^3 + 0 \times 2^2 + 0 \times 2 + v_0 =v, \label{xbigeru}
\end{align}
and we obtain 
\begin{equation}\label{ysmallth}
f(u,z) > v.
\end{equation}
When we reduce $x$ to $u$ and $y$ to $v$, by definition 
\ref{movefor3dimension} we have
\begin{equation}
v = \min(f(u,z),y), \nonumber
\end{equation}
and this contradicts (\ref{ysmallth}).

Therefore, let $\{u_1,v_1,z_1\} = \{0,1,1\}$. Using similar reasoning, we let $\{u_0,v_0,z_0\} = \{0,1,1\}$.

We then obtain the position $\{4,3,7\}$ such that $4 \oplus 3 \oplus 7 = 0$ and 
$v = 3 = \lfloor \frac{4+7}{3}\rfloor = f(u,z).$
\end{exam}

We define
\begin{equation}
S_t = \sum_{i=n-t}^n (x_i + z_i - ky_i) 2^i, \label{defofs},
\end{equation}
for $t = 0,1, \cdots, n$.

	
\begin{lemma}\label{lemmaforf} 
We have the following relationships between $f(x,z)$ and $S_n$.\\
$(a)$ 
	\begin{equation}
y = f(x,z) \nonumber 
	\end{equation}
if and only if $0 \leq S_n < k$.\\
$(b)$ 
	\begin{equation}
y > f(x,z) \nonumber 
	\end{equation}
if and only if $S_n < 0$.\\
$(c)$ 
	\begin{equation}
y < f(x,z) \nonumber 
	\end{equation}
if and only if $S_n \geq k$.
\end{lemma}
\begin{proof}
First, note that $S_n = \sum_{i=0}^n (x_i + z_i - ky_i) 2^i = x + z - ky$.\\
$(a)$
\begin{equation}
y = f(x,z) = \lfloor \frac{x+z}{k}\rfloor \nonumber 
\end{equation}
if and only if $y \leq \frac{x+z}{k} < y+1$ if and only if $0 \leq S_{n}= x+z-ky < k$.\\
$(b)$
\begin{equation}
y > f(x,z) = \lfloor \frac{x+z}{k}\rfloor \nonumber 
\end{equation}
if and only if $\frac{x+z}{k} < y$, which occurs if and only if $ S_{n}= x+z-ky < 0$.\\
We then obtain $(c)$ via $(a)$ and $(b)$.
\end{proof}

\begin{lemma}\label{lemma01}
Let $t \in Z_{\geq 0}$.
Suppose that for $i = n,n-1, \cdots, n-t$
	\begin{equation}
x_i \oplus y_i \oplus z_i=0. \label{oplus01}
	\end{equation}
There then exists an even number $a$ such that 
	\begin{equation}
S_t = a 2^{n-t}.\label{evenconditiont}
	\end{equation}	
\end{lemma}
\begin{proof}
Because $k$ is odd, by (\ref{oplus01}) $x_{i} + z_{i}-k y_{i}$ is even for $i = n, n-1, \cdots, n-t$, and therefore
we have (\ref{evenconditiont}).
\end{proof}

\begin{lemma}\label{lemma1}
Let $t \in Z_{\geq 0}$.
Suppose that for $i = n,n-1, \cdots, n-t$
	\begin{equation}
x_i \oplus y_i \oplus z_i=0 \label{oplus1}
	\end{equation}
and 
	\begin{equation}
S_t < 0.\label{negativeconditiont}
	\end{equation}	
Then, for any natural number $j$ such that $j > t$,

\begin{equation}
S_j < 0.\nonumber
\end{equation}
\end{lemma}
\begin{proof}
By Lemma \ref{lemma01}, (\ref{oplus1}),
and (\ref{negativeconditiont}),
\begin{equation}
S_t = a2^{n-t} \label{evencondition}
\end{equation}
for some even number $a$ such that $a \leq -2$.
Then, by (\ref{evencondition}), for any natural number $j$ such that $j > t$ the following holds:
\begin{align}
S_j = & S_t +   \sum_{i=n-j}^{n-t-1} (x_i + z_i - ky_i) 2^i  \nonumber \\
\leq & S_t + 2 \times \sum_{i=0}^{n-t-1} 2^i  \nonumber \\
\leq & (-2)2^{n-t} + 2 \times (2^{n-t} -1)= -2 < 0. \nonumber
\end{align}
\end{proof}

\begin{lemma}\label{lemma1b}
Let $t \in Z_{\geq 0}$.
Suppose that for $i = 0,1, \cdots, n-t$
	\begin{equation}
x_i \oplus y_i \oplus z_i=0 \label{oplus1b}
	\end{equation}
and 
	\begin{equation}
y \leq f(x,z).\label{ysmallerthanxz}
	\end{equation}
Then, 

	\begin{equation}
S_t \geq 0. \label{bigthan0}
	\end{equation}
	\end{lemma}
\begin{proof}
If
	\begin{equation}
S_t < 0, \label{smalthan0}
	\end{equation}
by (\ref{oplus1b}) and Lemma \ref{lemma1}, we have 
	\begin{equation}
S_n< 0. \nonumber
	\end{equation}
Then by $(b)$ of Lemma \ref{lemmaforf}, we have 
	\begin{equation}
y > f(x,z), \nonumber
	\end{equation}
and this contracts (\ref{ysmallerthanxz}). Therefore, (\ref{smalthan0}) is not true, and we have
(\ref{bigthan0}).
\end{proof}

\begin{lemma}\label{lemma2}
Let $t \in Z_{\geq 0}$.
If
	\begin{equation}
S_t \geq k2^{n-t},\label{greaterthank}
	\end{equation}	
then, for any natural number $j$ such that $j > t$,
\begin{equation}
S_j  \geq  k2^{n-j}. \label{greaterthanj}
\end{equation}
\end{lemma}
\begin{proof}
$S_j$ will be smallest when $\{x_i,y_i,z_i\}$ = $\{0,1,0\}$ for $i = n-t-1,n-t-2, \cdots, n-j.$
Therefore, it is sufficient to prove (\ref{greaterthanj}) for this case.
By (\ref{greaterthank}), for any natural number $j$ such that $j > t$,
\begin{align}
S_j = & S_t + \sum_{i=n-j}^{n-t-1} (x_i + z_i - ky_i) 2^i \nonumber \\
\geq & S_t  - k(2^{n-t-1}+2^{n-t-2} + \cdots +2^{n-j})    \nonumber \\
>  & k2^{n-t} - k(2^{n-t}-2^{n-j})  = k2^{n-j}. \nonumber
\end{align}
\end{proof}

\begin{lemma}\label{lemma3}
Let $t \in Z_{\geq 0}$.
Suppose that 
	\begin{equation}
0 \leq S_t \leq 2m \times 2^{n-t}.\nonumber
	\end{equation}	

Then, we have the following cases $(a)$ and $(b)$.\\
$(a)$ If $\{x_{n-t-1},y_{n-t-1},z_{n-t-1}\} = \{1,1,0\}$ or $\{0,1,1\}$,
then
\begin{equation}
S_{t+1}< 0.\nonumber
	\end{equation}
$(b)$ If $\{x_{n-t-1},y_{n-t-1},z_{n-t-1}\} = \{1,0,1\}$ or $\{0,0,0\}$,
then
\begin{equation}
0 \leq S_{t+1} < k \times 2^{n-t-1}.\nonumber
	\end{equation}
\end{lemma}
\begin{proof}
 $(a)$ If $\{x_{n-t-1},y_{n-t-1},z_{n-t-1}\} = \{1,1,0\}$ or $\{0,1,1\}$,
 \begin{align}
S_{t+1} = & S_t +2^{n-t-1}-k \times 2^{n-t-1} \nonumber \\
 \leq &  2m \times 2^{n-t} + 2^{n-t-1}-(4m+3)2^{n-t-1} \nonumber \\
 = & -2 \times 2^{n-t-1} <0.\nonumber
\end{align}

$(b)$  
If $\{x_{n-t-1},y_{n-t-1},z_{n-t-1}\} =  \{0,0,0\}$,
\begin{align}
0 & \leq S_{t+1} \nonumber \\
& = S_t \leq 4m \times 2^{n-t-1}  \nonumber \\
& < k \times 2^{n-t-1}.  \nonumber  
\end{align}
If
\begin{equation}
\{x_{n-t-1},y_{n-t-1},z_{n-t-1}\} = \{1,0,1\},
\end{equation}
\begin{align}
0 & \leq S_{t+1}  \nonumber \\
& = S_t + 2 \times 2^{n-t-1}  \nonumber \\
& \leq (4m+2)2^{n-t-1}  \nonumber \\
& < k \times 2^{n-t-1}.  \nonumber 
\end{align}

\end{proof}

\begin{lemma}\label{lemma4}
Let $t \in Z_{\geq 0}$.
Suppose that 
	\begin{equation}
(2m+2) \times 2^{n-t} \leq S_t < k \times 2^{n-t} \label{greaterthank22b}
	\end{equation}	
and
	\begin{equation}
x_i \oplus y_i \oplus z_i=0 \label{oplus001}
	\end{equation}
for $i = n,n-1, \cdots, n-t$.
Then, we have the following cases $(a)$ and $(b)$.\\
$(a)$ If $\{x_{n-t-1},y_{n-t-1},z_{n-t-1}\} = \{1,1,0\}$ or $\{0,1,1\}$,
then, 
\begin{equation}
0 \leq S_{t+1} < k \times 2^{n-t-1}.\nonumber
\end{equation}
$(b)$ If $\{x_{n-t-1},y_{n-t-1},z_{n-t-1}\}= \{1,0,1\}$ or $\{0,0,0\}$, 
then, 
\begin{equation}
S_{t+1} \geq k \times 2^{n-t-1}.\nonumber
\end{equation}
\end{lemma}
\begin{proof}
\noindent 
$(a)$ If $\{x_{n-t-1},y_{n-t-1},z_{n-t-1}\}= \{1,1,0\}$ or $\{0,1,1\}$, then
\begin{align}
S_{t+1} & = S_t +2^{n-t-1}-k2^{n-t-1} \nonumber \\
& = S_t - (4m+2)2^{n-t-1}. \nonumber 
\end{align}
By (\ref{greaterthank22b})
\begin{align}
& (2m+2)\times 2^{n-t} - (4m+2)2^{n-t-1} \nonumber \\
& \leq S_{t+1} = S_t -(4m+2)2^{n-t-1} \nonumber \\
& < (4m+3)2^{n-t}  -(4m+2)2^{n-t-1}.  \nonumber 
\end{align}

Therefore,
\begin{equation}
2 \times 2^{n-t-1} \leq S_{t+1} < (4m+4)2^{n-t-1}.  \label{eqa1}
\end{equation}
By (\ref{oplus001}) and Lemma \ref{lemma01}, $S_{t+1} = a2^{n-t-1}$ for some even integer $a$, and therefore by (\ref{eqa1}) we have 

\begin{equation}
 0< S_{t+1} \leq  (4m+2)2^{n-t-1} < k\times 2^{n-t-1}.  \nonumber 
\end{equation}
$(b)$
 If $(x_{n-t-1},x_{n-t-1},x_{n-t-1}) = \{1,0,1\}$ or $\{0,0,0\}$, then by (\ref{greaterthank22b})
\begin{equation}
 k \times 2^{n-t-1} < (4m+4)2^{n-t-1} \leq S_t \leq S_{t+1}. \nonumber 
\end{equation}
\end{proof}

\begin{lemma}\label{fromNtoPforh}
We assume that
\begin{equation}
x \oplus y \oplus z \neq 0 \label{nimsumno0}
\end{equation}
and 
\begin{equation}\label{inequalityyz}
y \leq f(x,z).
\end{equation}
Then, at least one of the following statements is true:\\
$(1)$ $u\oplus y \oplus z= 0$ for some $u\in Z_{\geq 0}$ such that  $u<x$;\\
$(2)$ $u \oplus v \oplus z= 0$  for some $u, v \in Z_{\geq 0}$ such that  $u < x,v < y$ and $v=f(u,z)$;\\
$(3)$ $x\oplus v\oplus z= 0$ for some $v\in Z_{\geq 0}$ such that  $v<y$;\\
$(4)$ $x\oplus y\oplus w= 0$ for some $w\in Z_{\geq 0}$ such that  $w<z$ and $y \leq f(x,w)$;\\
$(5)$ $x\oplus v\oplus w= 0$ for some $v,w \in Z_{\geq 0}$ such that  $v<y,w <z$ and $v=f(x,w)$.\\
\end{lemma}
\begin{proof}
Let $x= \sum\limits_{i = 0}^n {{x_i}} {2^i}$ and $y= \sum\limits_{i = 0}^n {{y_i}} {2^i}$, and $z= \sum\limits_{i = 0}^n {{z_i}} {2^i}$. 

If $n=0$, we have $x,z \leq 1$, and  $y \leq f(x,z) = \lfloor \frac{x+z}{k}\rfloor = 0$. Then, by (\ref{nimsumno0}), we have
$\{x,y,z\} = \{1,0,0\}$ or $\{0,0,1\}$. In this case we obtain $(1)$ for $\{u,y,z\} = \{0,0,0\}$ or $(4)$ for $\{x,y,w\} = \{0,0,0\}$ by reducing $x=1$ to $u=0$ or reducing $z=1$ to $w =0$.

Next, we assume that $n \geq 1$ and that there exists a non-negative integer $t$ such that 
\begin{equation}\label{zerotonminusut1}
x_i \oplus y_i \oplus z_i = 0 
\end{equation}
for $i = n,n-1,...,n-t+1$ and  
\begin{equation}\label{xyznminush}
x_{n-t} \oplus y_{n-t} \oplus z_{n-t} \neq 0.
\end{equation}
Let $S_k = \sum_{i=n-k}^n (x_i + z_i - ky_i) 2^i$ for $k=0,1,\cdots ,t-1$, and we may then 
define $S_k$ for $k=t,t+1,\cdots ,n$.

By (\ref{zerotonminusut1}), (\ref{inequalityyz}), and Lemma \ref{lemma1b},
we have
\begin{equation}\label{xyznminush22}
S_{t-1} \geq 0.
\end{equation}

We then have three cases.\\
\underline{Case $(1)$} 
Suppose that $\{x_{n-t},y_{n-t},z_{n-t}\}=\{1,0,0\}$. We reduce $x$ to $u$ and for $i = 0,1, \cdots, t-1$ let 
\begin{equation}
u_{n-i} = x_{n-i}, \nonumber 
\end{equation}
and we define $u_{n-i}$ for $i = t,t+1, \cdots, n$ using an inductive method with the following steps $[I]$ and $[II]$.\\
Step $[I]$ 
Let $u_{n-t}=0$. Then, 
\begin{equation}
u= \sum\limits_{i = 0}^n {{u_i}} {2^i}<  \sum\limits_{i = 0}^n {{x_i}} {2^i}=x.\nonumber
\end{equation}
Because $u_{n-t}=0$ and $y_{n-t}=z_{n-t}=0$, by (\ref{xyznminush22}) we have
\begin{equation}\label{xyznminush22b}
S_{t} = S_{t-1}+(u_{n-t}+y_{n-t}-kz_{n-t})2^{n-t} = S_{t-1} \geq 0.\nonumber
\end{equation}
We then consider two subcases $(1.1)$ and $(1.2)$ according to the value of $S_{t}$.\\
\underline{Subcase $(1.1)$} 
We suppose that 
\begin{equation}
0 \leq S_{t} \leq 2m \times 2^{n-t}.\label{from0to2t}
\end{equation}
We then have two subsubcases for two possible values of $z_{n-t-1}$. \\
\underline{Subsubcase $(1.1.1)$} 
Suppose that $z_{n-t-1} = 0$. We have two subsubsubcases for two possible values of $y_{n-t-1}$. \\
\underline{Subsubsubcase $(1.1.1.1)$} 
If $y_{n-t-1} = 0$, let 
\begin{equation}
\{u_{n-t-1}, y_{n-t-1},z_{n-t-1}\}=\{0,0,0\} \nonumber
\end{equation}
and 
\begin{equation}
S_{t+1}= S_{t}+(u_{n-t-1}+z_{n-t-1}-ky_{n-t-1})2^{n-t-1}. \nonumber
\end{equation}
Then, by Lemma \ref{lemma3} and (\ref{from0to2t})
\begin{equation}\label{condition01}
0 \leq S_{t+1} < k2^{n-t-1}.
\end{equation}
Then, we begin Step $[II]$ with (\ref{condition01}) while knowing that $y$ has not been reduced.\\
\underline{Subsubsubcase $(1.1.1.2)$} 
If $y_{n-t-1} = 1$, let $v_{n-t-1}=0<y_{n-t-1}$.
Then, we have
\begin{equation}
 \sum\limits_{i = 0}^n {{v_i}} {2^i} < \sum\limits_{i = 0}^n {{y_i}} {2^i} \nonumber
\end{equation}
for any values of $v_{i}$ for $i = 0,1, \cdots, n-t-1$. In this subsubsubcase, 
we reduce $y$ to $v$ by reducing $x$ to $u$.
For a concrete example of reducing $y$ to $v$ by reducing $x$ to $u$, see $(ii)$ and $(iii)$ in Example \ref{chococute}.

Let 
\begin{equation}
\{u_{n-t-1}, v_{n-t-1},z_{n-t-1}\}=\{0,0,0\} \nonumber
\end{equation}
and 
\begin{equation}
S_{t+1}= S_{t}+(u_{n-t-1}+z_{n-t-1}-kv_{n-t-1})2^{n-t-1},\nonumber
\end{equation}
then by Lemma \ref{lemma3} and (\ref{from0to2t})
\begin{equation}
0 \leq S_{t+1} < k2^{n-t-1}.\label{condition02}
\end{equation}
Then, we begin Step $[II]$ with (\ref{condition02}) while knowing that we have reduced $y$ to $v$.\\
\underline{Subsubcase $(1.1.2)$} 
Suppose that $z_{n-t-1} = 1$. We have two subsubsubcases for two possible values of $y_{n-t-1}$. \\
\underline{Subsubsubcase $(1.1.2.1)$} 
If $y_{n-t-1} = 0$, let
\begin{equation}
\{u_{n-t-1}, y_{n-t-1},z_{n-t-1}\}=\{1,0,1\} \nonumber
\end{equation}
and 
\begin{equation}
S_{t+1}= S_{t}+(u_{n-t-1}+z_{n-t-1}-ky_{n-t-1})2^{n-t-1}, \nonumber
\end{equation}
then, by Lemma \ref{lemma3} and (\ref{from0to2t})
\begin{equation}
0 \leq S_{t+1} < k 2^{n-t-1}.\label{condition03}
\end{equation}
Then, we begin Step $[II]$ with (\ref{condition03}) while knowing that $y$ has not been reduced.\\
\underline{Subsubsubcase $(1.1.2.2)$} 
If $y_{n-t-1} = 1$, let $v_{n-t-1}=0<y_{n-t-1}$.
Then, we have
\begin{equation}
 v=\sum\limits_{i = 0}^n {{v_i}} {2^i} < \sum\limits_{i = 0}^n {{y_i}} {2^i}=y \nonumber
\end{equation}
for any values of $v_{i}$ for $i = 0,1, \cdots, n-t-1$. In this subsubsubcase, 
we reduce $y$ to $v$ by reducing $x$ to $u$. For a concrete example of reducing $y$ to $v$ by reducing $x$ to $u$, see $(ii)$ and $(iii)$ in Example \ref{chococute}.

Then, let 
\begin{equation}
\{u_{n-t-1}, v_{n-t-1},z_{n-t-1}\}=\{1,0,1\} \nonumber
\end{equation}
and 
\begin{equation}
S_{t+1}= S_{t}+(u_{n-t-1}+z_{n-t-1}-kv_{n-t-1})2^{n-t-1}, \nonumber
\end{equation}
then, by Lemma \ref{lemma3} and (\ref{from0to2t})
\begin{equation}
0 \leq S_{t+1} < k 2^{n-t-1}. \label{condition04}
\end{equation}
Then, we begin Step $[II]$ with (\ref{condition04}) while knowing that we have reduced $y$ to $v$.\\
\underline{Subcase $(1.2)$} 
We suppose that 
\begin{equation}
(2m+2)2^{n-t} \leq S_{t} < k2^{n-t}.\label{biggerth2m2}
\end{equation}
We have two subsubcases for two possible values of $z_{n-t-1}$. \\
\underline{Subsubcase $(1.2.1)$} 
Suppose that $z_{n-t-1} = 0$. We have two subsubsubcases for two possible values of $y_{n-t-1}$. \\
\underline{Subsubsubcase $(1.2.1.1)$} 
If $y_{n-t-1} = 1$, let 
\begin{equation}
\{u_{n-t-1}, y_{n-t-1},z_{n-t-1}\}=\{1,1,0\} \nonumber
\end{equation}
and 
\begin{equation}
S_{t+1}= S_{t}+(u_{n-t-1}+z_{n-t-1}-ky_{n-t-1})2^{n-t-1},\nonumber
\end{equation}
then, by Lemma \ref{lemma4} and (\ref{biggerth2m2})
\begin{equation}
0 \leq S_{t+1} < k2^{n-t-1}.\label{condition05}
\end{equation}
Then, we begin Step $[II]$ with (\ref{condition05}) while knowing that $y$ has not been reduced.\\
\underline{Subsubsubcase $(1.2.1.2)$} 
If $y_{n-t-1} = 0$, let 
\begin{equation}
\{u_{n-t-1}, y_{n-t-1},z_{n-t-1}\}=\{0,0,0\}.\nonumber
\end{equation}
By Lemma \ref{lemma4} and (\ref{biggerth2m2})
\begin{equation}
 S_{t+1} \geq k2^{n-t-1}.\label{condition06}
\end{equation}
Then, we begin Step $[II]$ with (\ref{condition06}) while knowing that $y$ has not been reduced.\\
\underline{Subsubcase $(1.2.2)$} 
Suppose that $z_{n-t-1} = 1$. We have two subsubsubcases for two possible values of $y_{n-t-1}$. \\\\
\underline{Subsubsubcase $(1.2.2.1)$} 
If $y_{n-t-1} = 1$, let 
\begin{equation}
\{u_{n-t-1}, y_{n-t-1},z_{n-t-1}\}=\{0,1,1\}.\nonumber
\end{equation}
and 
\begin{equation}
S_{t+1}= S_{t}+(u_{n-t-1}+z_{n-t-1}-ky_{n-t-1})^{n-t-1},\nonumber
\end{equation}
then, by Lemma \ref{lemma4} and (\ref{biggerth2m2})
\begin{equation}
0 \leq S_{t+1} < k2^{n-t-1}.\label{condition07}
\end{equation}
Then, we begin Step $[II]$ with (\ref{condition07}) and the fact that $y$ has not been reduced.\\
\underline{Subsubsubcase $(1.2.2.2)$} 
If $y_{n-t-1} = 0$, let 
\begin{equation}
\{u_{n-t-1}, y_{n-t-1},z_{n-t-1}\}=\{1,0,1\}.\nonumber
\end{equation}
By Lemma \ref{lemma4} and (\ref{biggerth2m2})
\begin{equation}
 S_{t+1} \geq k 2^{n-t-1}.\label{condition08}
\end{equation}
Then, we begin Step $[II]$ with (\ref{condition08}) while knowing that $y$ has not been reduced.\\
\underline{Case $(2)$} 
We suppose that $\{x_{n-t},y_{n-t},z_{n-t}\}=\{0,0,1\}$.
Then, we can use the same method used for Case $(1)$.\\
\underline{Case $(3)$} 
We suppose that $y_{n-t}=1$.
Let $v_{n-t}=0 < y_{n-t}$ and $v_{i}=x_{i}+z_{i}$ (mod 2) for $i=n-t-1, \cdots, 0$. Then, we have $x\oplus v\oplus z= 0$ and $v < y \leq f(x,z)$, and we have $(3)$ of this lemma. In this case, we do not need Step $[II]$\\
Step $[II]$. We have two cases.\\
\underline{Case $(1)$} This is a sequel to Case $(1)$ of Step $[I]$.
Here, the procedure consists of three subcases.\\
\underline{Subcase $(1.1)$} Suppose that $S_{t+1} \geq k 2^{n-t-1}$.
In this case, $y$ is not reduced to $v$ in the last procedure, i.e., Step $[I]$.
By Lemma \ref{lemma2}, we let $u_{i}=y_i + z_i \ (\mod 2)$ for $i=n-t-2, \cdots, 0$ without affecting the values of $y_{i}$ for $i=n-t-2, \cdots, 0$. Then, we have $(1)$ for this lemma.\\
\underline{Subcase $(1.2)$} Suppose that $0 \leq S_{t+1} < k 2^{n-t-1}$ and 
 $y$ was reduced to $v$ in Step $[I]$.
 Then, we choose the values of $u_{n-i}, v_{n-i}$ for $i = t+2, t+3, \cdots, n$ such that 
 $0 \leq S_{i} < k 2^{n-i}$ by the following $(a)$ or $(b)$.\\
$(a)$ For $i \geq t+1$, if $0 \leq S_{i} < 2m \times 2^{n-i}$, then we let $\{u_{n-i-1},v_{n-i-1},z_{n-i-1} \}$
 $= \{0,0,0\}$ or $\{1,0,1\}$ when $z_{n-i-1}=0$ or  $z_{n-i-1}=1$, respectively.
 Then, by Lemma \ref{lemma3}, we have $0 \leq S_{i+1} < k 2^{n-i-1}$.\\
$(b)$ For $i \geq t+1$, if $ S_{i} \geq (2m+2) \times 2^{n-i}$, then we let $\{u_{n-i-1},v_{n-i-1},z_{n-i-1} \}$
 $= \{1,1,0\}$ or $\{0,1,1\}$ when $z_{n-i-1}=0$ or $z_{n-i-1}=1$, respectively.
 Then, by Lemma \ref{lemma4}, we have $0 \leq S_{i+1} < k 2^{n-i-1}$.
 
 Therefore, for $i = t+2, \cdots ,$ we have $0 \leq S_{i} < k 2^{n-i}$, and finally 
 we have $0 \leq S_{n} < k 2^{n-n}=k$. We then have 
 $v = f(u,z)$ and $u \oplus v \oplus z = 0$; therefore, we have $(5)$ of this lemma.\\
\underline{Subcase $(1.3)$} Suppose that $0 \leq S_{t+1} < k 2^{n-t-1}$ and 
 $y$ was not reduced to $v$ during the last procedure. In this case, we use the same method as in step $[I]$. \\
 \underline{Case $(2)$} This is a sequel to Case $(2)$ of Step $[I]$.
Then, we can use the same method used for Case $(1)$ of Step $[I]$.
\end{proof}

\begin{lemma}\label{fromPtoNforh}
We assume that
 $x \oplus y \oplus z = 0$ and 
\begin{equation}\label{inequalityyz2}
y \leq f(x,z).
\end{equation}
Then, the following hold:\\
$(1)$ $u\oplus y \oplus z \neq 0$ for any $u\in Z_{\geq 0}$ such that  $u<x$;\\
$(2)$ $u \oplus v \oplus z \neq  0$ for any $u, v \in Z_{\geq 0}$ such that  $u < x,v < y$, and $v=f(u,z)$;\\
$(3)$ $x\oplus v\oplus z \neq 0$ for any $v\in Z_{\geq 0}$ such that  $v<y$;\\
$(4)$ $x\oplus y\oplus w \neq 0$ for any $w\in Z_{\geq 0}$ such that  $w<z$ and $y \leq f(x,w)$;\\
$(5)$ $x\oplus v\oplus w \neq 0$ for any $v,w \in Z_{\geq 0}$ such that  $v<y,w <z$, and $v=f(x,w)$.
\end{lemma}
\begin{proof} 
If $x \oplus y \oplus z = 0$, a positive value of the nim-sum is obtained by changing the value of one of $x,y,z$.
Therefore, we have 
$(1)$, $(3)$, and $(4)$.

Next, we prove $(2)$. The only way to have 
$u \oplus v \oplus z =  0$ for some $u, v \in Z_{\geq 0}$ such that  $u < x,v < y$ and 
\begin{equation}\label{equalityfor}
v=f(u,z)
\end{equation}
is to reduce $\{x_{n-t},y_{n-t},z_{n-t}\} = \{1,1,0\}$ to $\{u_{n-t},v_{n-t},z_{n-t}\} = \{0,0,0\}$. We consider two cases.\\
\underline{Case $(1)$} Suppose that $0 \leq S_{t-1} \leq 2m \times 2^{n-t+1}$. Then, for $\{x_{n-t},y_{n-t},z_{n-t}\} = \{1,1,0\}$ by Lemma \ref{lemma3}, we have $S_{t+1}< 0$. Then, by Lemmas \ref{lemma1} and \ref{lemmaforf}, we have $y > f(x,y)$. This contradicts (\ref{inequalityyz2}).\\
\underline{Case $(2)$} Suppose that $ S_{t-1} \geq (2m+2) \times 2^{n-t+1}$. For $\{u_{n-t},v_{n-t},z_{n-t}\} = \{0,0,0\}$ by Lemma \ref{lemma4}, 
$S_{t} \geq k \times 2^{n-t}$; therefore, by Lemma \ref{lemma2}, $S_n  \geq  k$. Using Lemma \ref{lemmaforf}, we then have $y < f(u,v)$.
This contradicts (\ref{equalityfor}).

Similarly we can prove $(5)$.
\end{proof}

\begin{theorem}\label{theoremforoddk}
Let $f(x,z)  = \lfloor \frac{x+z}{k}\rfloor$ for $k = 4m+3$. Then, the
chocolate bar $CB(f,x,y,z)$ is a $\mathcal{P}$-position if and only if 
	\begin{equation}
		 x \oplus y \oplus z=0. \label{nscondtionfornim0}
	\end{equation}
\end{theorem}
\begin{proof}
Let $A_k=\{\{x,y,z\}:x\oplus y \oplus z = 0\}$ and $B_k =\{\{x,y,z\}: x\oplus y \oplus z \neq 0\}$.
 If we begin the game with a position $\{x,y,z\}\in A_{k}$, then using Theorem \ref{fromPtoNforh}, any option leads to a position  $\{p,q,r\} \in B_k$. 
From this position $\{p,q,r\}$ by Theorem \ref{fromNtoPforh}, our opponent can choose an appropriate option that leads to a position in $A_k$. Note that any option reduces some of the numbers in the coordinates. In this way, our opponent can always reach a position in $A_k$, and finally, they win by reaching $\{0,0,0\}\in A_{k}$. Therefore, $A_k$ is the set of $\mathcal{P}$-positions.

If we begin the game with a position $\{x,y,z\}\in B_{k}$, then by  Theorem \ref{fromNtoPforh}, we can choose an appropriate option that leads to a position $\{p,q,r\}$ in $A_k$. From $\{p,q,r\}$, any option chosen by our opponent leads to a position in $B_k$. In this way, we win the game by reaching $\{0, 0, 0\}$. Therefore, $B_k$ is the set of $\mathcal{N}$-positions.\\
\end{proof}

\section{othercases}\label{others}
The result shown in Section \ref{sub4mone} depends on the assumption that $k$ is odd, but it seems that a similar result can be proven for an even number $k$ with a restriction on the size of $x,z$.

The authors discovered the following conjecture via calculations using the computer algebra system Mathematica, but they have not managed to prove it.

\begin{conjecture}\label{theoremmanabe}
Let $f(x,z) = \lfloor \frac{x+z}{k}\rfloor$ for $k = 2^{a+2}m+2^{a+1}$ and $x,z \leq (2^{2a+2}-2^{a+1})m+2^{2a+1}-1$, where  $a,m \in Z_{\ge 0}$. Then, the chocolate bar $CB(f,x,y,z)$ is a $\mathcal{P}$-position if and only if 
	\begin{equation}
		 x \oplus y \oplus z=0. \nonumber
	\end{equation}
\end{conjecture}

\begin{rem}
If we compare Theorem \ref{theoremforoddk} and Conjecture \ref{theoremmanabe}, it seems very difficult to obtain the necessary and sufficient condition for Question 2. 
\end{rem}

The authors also have the following conjecture that also 
may be derived via calculations using the computer algebra system Mathematica.

\begin{conjecture}
Let $f(x,z) = \lfloor \frac{x+z}{k}\rfloor$ for $k = 4m + 1$. Then, the chocolate bar $CB(f,x,y,z)$ is a $\mathcal{P}$-position if and only if 
	\begin{equation}
	(x+1) \oplus y \oplus (z+1)=0.  \nonumber
	\end{equation}
\end{conjecture}

\section*{Acknowledgements}
We are indebted to Shouei Takahasi and Taishi Aono. Although not the primary authors, their contributions were significant. We would like to thank Editage (www.editage.com) for English language editing. This work was
supported by Grant-in-Aid for Scientific Research of Japan.

\end{document}